%% file: _a_main.tex
\documentclass[11pt]{article}

\usepackage[english]{babel}

\usepackage[letterpaper,top=2cm,bottom=2cm,left=3cm,right=3cm,marginparwidth=1.75cm]{geometry}

\usepackage{amsmath}
\usepackage{amssymb}
\usepackage{graphicx}
\usepackage[colorlinks=true, allcolors=blue]{hyperref}
\usepackage[dvipsnames]{xcolor}
\usepackage{todonotes,setspace}

\usepackage{centernot}

\newcounter{mycomment}
\newcommand{\mycomment}[2][]{%
   \refstepcounter{mycomment}%
   \ifcase\numexpr \themycomment- 4*((\themycomment+3)/4 -1) \relax
   {%
      \setstretch{0.7}
      \todo[color=SeaGreen,size=\scriptsize]{%
         \textbf{[\uppercase{#1}\themycomment]:} #2}%
   }
   \or
   {%
      \setstretch{0.7}
      \todo[color=Salmon,size=\scriptsize]{%
         \textbf{[\uppercase{#1}\themycomment]:} #2}%
   }
   \or
   {%
      \setstretch{0.7}
      \todo[color=Thistle,size=\scriptsize]{%
         \textbf{[\uppercase{#1}\themycomment]:} #2}%
   }
   \else
   {%
      \setstretch{0.7}
      \todo[color=Tan,size=\scriptsize]{%
         \textbf{[\uppercase{#1}\themycomment]:} #2}%
   }
   \fi
}

\usepackage{amsthm}

\theoremstyle{definition}

\newtheorem{question}{Question}[section]
\newtheorem{conjecture}{Conjecture}[section]
\newtheorem{corollary}{Corollary}[section]
\newtheorem{theorem}{Theorem}[section]
\newtheorem{remark}{Remark}[section]
\newtheorem{lemma}{Lemma}[section]
\usepackage{mathtools}
\usepackage{tkz-euclide}

\usepackage{orcidlink}

\newtheorem{definition}{Definition}
\newtheorem{proposition}{Proposition}

\usepackage{amsmath}
\usepackage{amsfonts}
\usepackage{bbm}
\usepackage{pgfplots}
\pgfplotsset{compat=1.18}
\usepackage{tikz}
\usepackage{enumerate}

\title{Induced Distributions from Generalized Unfair Dice} 
\author{Douglas T. Pfeffer\,\orcidlink{0000-0001-8287-2686}, J. Darby Smith\,\orcidlink{0000-0002-3643-0868}, William Severa\,\orcidlink{0000-0002-8740-220X}}
\date{}

\begin{document}
\maketitle

\begin{abstract}
In this paper we analyze the probability distributions associated with rolling  (possibly unfair) dice infinitely often.  Specifically, given a $q$-sided die, if $x_i\in\{0,\ldots,q-1\}$ denotes the outcome of the $i^{\text{th}}$ toss, then the distribution function is $F(x)=\mathbb{P}[X\leq x]$, where $X = \sum_{i=1}^\infty x_i q^{-i}$. We show that $F$ is singular and establish a piecewise linear, iterative construction for it. We investigate two ways of comparing $F$ to the fair distribution---one using supremum norms and another using arclength. In the case of coin flips, we also address the case where each independent flip could come from a different distribution. In part, this work aims to address outstanding claims in the literature on Bernoulli schemes.  The results herein are motivated by emerging needs, desires, and opportunities in computation to leverage physical stochasticity in microelectronic devices for random number generation.
\end{abstract}

\newcommand{\Addresses}{{
  \bigskip
  \footnotesize

\setlength{\parindent}{0cm}
    \noindent Douglas T. Pfeffer\,\orcidlink{0000-0001-8287-2686}\par\nopagebreak\textsc{Department of Mathematics, University of Tampa, 401 W. Kennedy Blvd. Tampa, FL 33606}\par\nopagebreak
  e-mail: \texttt{dpfeffer@ut.edu}

  \medskip

  J. Darby Smith\,\orcidlink{0000-0002-3643-0868}\par\nopagebreak  \textsc{Neural Exploration and Research Laboratory, Center for Computing Research, Sandia National Laboratories, 1515 Eubank Blvd.~SE, Albuquerque, NM 87123}\par\nopagebreak
  e-mail: \texttt{jsmit16@sandia.gov}

  \medskip

  William Severa\,\orcidlink{0000-0002-8740-220X}\par\nopagebreak  \textsc{Neural Exploration and Research Laboratory, Center for Computing Research, Sandia National Laboratories, 1515 Eubank Blvd.~SE, Albuquerque, NM 87123}\par\nopagebreak
  e-mail: \texttt{wmsever@sandia.gov}
}}

\section{Introduction\label{sec:intro}}
\input{_b_intro}

\section{Existing results on Bernoulli schemes}\label{sec:bernoulli}
\input{_c_bernoulli}

\section{Results on infinite (unfair) dice rolling.} \label{sec:qsideddie}
\input{_d_main_dice}

\section{Comparisons of distributions to uniform.}\label{sec:arclength}
\input{_e_comparisons}

\section{Discussion \label{sec:conclusion}}
\input{_f_discussion}

\section*{Acknowledgements}
This article has been authored by an employee of National Technology \& Engineering Solutions of Sandia, LLC under Contract No. DE-NA0003525 with the U.S. Department of Energy (DOE). The employee owns all right, title and interest in and to the article and is solely responsible for its contents. The United States Government retains and the publisher, by accepting the article for publication, acknowledges that the United States Government retains a non-exclusive, paid-up, irrevocable, world-wide license to publish or reproduce the published form of this article or allow others to do so, for United States Government purposes. The DOE will provide public access to these results of federally sponsored research in accordance with the DOE Public Access Plan \url{https://www.energy.gov/downloads/doe-public-access-plan}.  The  authors  acknowledge  support  from  the  DOE  Office  of  Science  (ASCR/BES)   Microelectronics   Co-Design   project   COINFLIPS.
\bibliographystyle{alpha}
\bibliography{flip}

\Addresses

\end{document}

%% file: _b_intro.tex
Contemporary computing approaches are dominated by deterministic operations.  In terms of both the algorithmic approach and the underlying computing devices, determinism is deeply woven into our computing mindset.  At the device level, stochastic behavior is often seen as a defect. Noise and fluctuations have been eliminated or constrained wherever possible.  This is often beneficial; resistance to noise is one of the key benefits of \textit{digital} electronics.  However, a direct consequence is that our everyday computers are deterministic.

Of course, randomness plays a role in many algorithms, including those from scientific computing and cryptography.  For example, particle methods and other probabilistic approaches are often applied to high-dimensional physics problems where direct numerical solutions can be intractable.  However, there is an inherent misalignment between stochastic behavior and deterministic hardware.

Well-distributed and difficult-to-predict numbers can be generated by a Pseudo-Random Number Generator (pRNG).  These methods (generally) take a \textit{seed} value and generate a sequence of corresponding numbers through iteration.  The sequence of values can appear to be random, but are entirely determined by the seed and the pRNG.  The quality of the `random' numbers is dependent on the quality of the pRNG.  While this method is sufficient for many applications, deficiencies in either the seed setting or the pRNG can be disastrous.\footnote{We provide two examples, though we suggest the reader to explore the fascinating world of deficient pRNGs.  The first is of historical notoriety: The RANDU generator's iterates in three dimensions fall along planes~\cite{markowsky2014sad}, making predictions trivially easy in this case. The second is that system time is a common seed value.  Knowing the system time at seed setting allows players to predict future events in games such as Pok\'emon~\cite{pika}.}

Sources of noise can be used to help improve the quality of pRNG number generation.  Small timing differences in keyboard input and mouse input or fluctuations in measured quantities, such as WiFi signal, can be used.  In modern approaches, these noisy signals feed what is called an \textit{entropy pool}.  This entropy pool (e.g.~via \texttt{/dev/random/} on Linux) can then be combined, hashed, and otherwise manipulated to produce yet more unpredictable ``random'' numbers.

Unfortunately, this entropy pool approach has three main challenges:
\begin{enumerate}
\item The sources of noise may not be truly random.
\item The pRNGs still produce non-random numbers.
\item The entropy pool can be depleted.
\end{enumerate}

We believe these motivate the study of probabilistic computing devices and, consequently, the study of how to best use a naturally stochastic computing device.  This motivation is shared by those in the computing field, where probabilistic devices and true Random Number Generators (tRNGs) are an area of active study \cite{misra2022probabilistic,chowdhury2023accelerated,chowdhury2023full,aadit2022massively,kaiser2021probabilistic}. Certain devices, such as magnetic tunnel junctions \cite{rehm2022stochastic,liu2022random} or tunnel diodes \cite{bernardo2017extracting}, can be made to behave as an effective Bernoulli random variable (hereafter, a coin flip). These devices have some probability $p$ of returning 1 and a probability $1-p$ of returning 0. 

Simple random variables have great utility and can be exploited to return samples from a variety of distributions~\cite{gryszka2021biased,flegal2012exact}. Furthermore, when used as input to novel and emerging hardware, like spiking neurons in a neuromorphic computer~\cite{schuman2022opportunities,young2019review,aimone2021roadmap}, such simple stochastic input can be used to find approximate solutions to NP-hard problems such as MAXCUT~\cite{theilman2022stochastic,theilman2023goemans}.

These formulations tend to model coin flips as precisely that---A two-outcome, even draw.  We note that these devices conceivably behave not as just coins but as $q$-sided dice. Indeed, current devices considered for other purposes can be in one of five states~\cite[Fig. 2e]{leonard2023stochastic,leonard2022shape}. Though the bulk of current investigations examine coin-like devices, such as $p$-bits~\cite{camsari2017stochastic}, in the future we may find that dice-like devices are more attractive for physical reasons (more stable or more efficient) or perhaps even for computational ones (reduced burden or circuitry).  An additional complication is that devices may not be fair, and it is critical to understand any departure from uniform in the general setting.

To address these questions, we examine binary encoded distributions from Bernoulli coin flips and $q$-ary encoded distributions from $q$-sided die rolls. We revisit classic results in the Bernoulli case for a class of singular measures and state a `folk theorem' on the extension to $n$-sided dice in Section \ref{sec:bernoulli}.  We establish this extension and provide a proof of the `folk theorem' in Section \ref{sec:qsideddie}.  Finally, in Section \ref{sec:arclength}, we seek to address the question `Given an unfair $q$-sided die, how does its CDF compare to the uniform, fair die case?' To this end, we provide two comparative tools one can use.  In Section \ref{supcompare}, Theorem \ref{theorem: uniformcomp} provides an easily calculable upper bound on $\|x-F(x)\|_\infty$, and in Section \ref{comparearclength}, Theorem \ref{theorem:arclength} provides a formula for calculating the arclength of $F(x)$ after finitely many dice rolls (which can then be compared to the fair arclength of $\sqrt{2}$).  These formulas answer curiosity in and of themselves as well as stand to provide a basis for error analysis in the future of probabilistic computing. We conclude our effort with a discussion in Section \ref{sec:conclusion} that seeks to frame these results in the context of novel mathematical directions and future stochastic computing.

%% file: _c_bernoulli.tex
Consider a series of Bernoulli outcomes $x_i$ on $\{0,1\}$. As coin flips, we will call $1$ `heads' and $0$ `tails'. In the independent and identically distributed (i.i.d.) case, let $\mathbb{P}[x_i=0]=p_0$ and, consequently, $\mathbb{P}\left[x_i=1\right]=1-p_0=p_1$.  If $p_0=0.5$, we say we are flipping a `fair' coin. Given $n$ outcomes, define
\[
X_n := \frac{x_1}{2^1} + \frac{x_2}{2^2} + \cdots + \frac{x_n}{2^n} = \sum_{i=1}^n \frac{x_i}{2^i}.
\]
Each $X_n$ is a decimal number in $[0,1]$ formed from a binary encoding of $n$ outcomes $\{x_i\}_{i=1}^n$.

Let $X := \lim_{n\to\infty} X_n$ be the encoded value in $[0,1]$ obtained by flipping our coin infinitely many times.  We ask, given $y\in[0,1]$, what is $\mathbb{P}[X\leq y]$? That is, what is the cumulative distribution function (CDF) of $X$?

Consider first the case that $p=0.5$. For a finite number of flips $N$, the probability of getting any single value of $X$ is the probability of getting a particular string of outcomes. In this fair case, that probability is $1/2^N$. Extending this probability mass function to a probability density function on the real line, we need to divide this probability by the width of the unit of mass it represents. Given the binary encoding, this width is $1/2^N$. Hence the probability density induced on the real line by $N$ flips is 1. The limit in $N$ is still 1, and the associated probability density function (PDF) of $X$ in the fair case is the uniform PDF.  Therefore the CDF of $X$ in the fair case is given by 
\begin{equation}\label{eq:uniform}
F(x) = \int_{-\infty}^x \mathbbm{1}_{[0,1]} \mathop{d\mu} = \begin{cases}0 &\text{if} \ x<0\\
x &\text{if} \ 0\leq x \leq 1\\
1 &\text{if} \ x>1
\end{cases}.
\end{equation}
We remark that the uniform measure (the PDF) on $[0,1]$ coincides with Lebesgue measure restricted to $[0,1]$.

What happens when $p\neq 0.5$?~\cite[Example 31.1]{billingsley} shows that the distribution turns out to be \textit{singular}---a probability distribution concentrated on a set of Lebesgue measure zero. This is observed by showing that the cumulative distribution function $F(x)$ for $X$ in the non-fair case is continuous, strictly increasing, and has $F'(x)=0$ almost everywhere. A related concept is the Cantor distribution, a probability distribution whose cumulative distribution function is the Cantor function.

After demonstrating that the cumulative distribution function, $F(x)$, for the unfair coin is singular, Billingsley establishes a recursion definition for it. With $\mathbb{P}[x_i=0]=p_0$ and $p_1 = 1-p_0$,
\begin{equation} \label{unfairrecursive}
    F(x) = 
    \begin{cases} 
    p_0F(2x) &\text{if} \ 0\leq x \leq \textstyle \frac{1}{2}\\
    p_0 + p_1F(2x-1) &\text{if} \ \textstyle \frac{1}{2} \leq x \leq 1 
    \end{cases}.
\end{equation}
As examples, we graph this recursive function for four different coins $p_0 = .15, .25, .40, \text{ and } .49$
in Figure \ref{fig:unfair_example}.

\begin{figure}
    \centering
    \includegraphics[width=.65\textwidth]{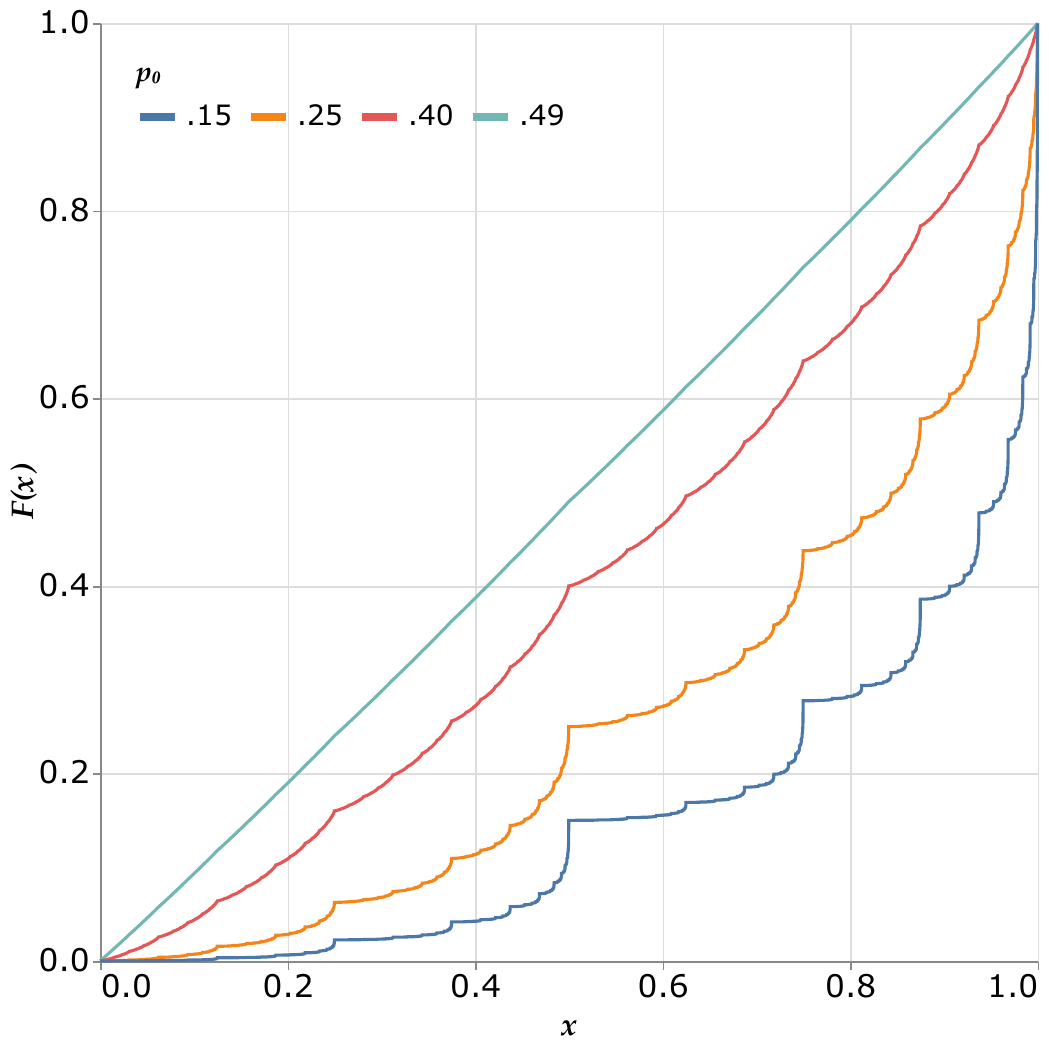}
    \caption{The CDF $F(x)$ shown for various unfair coins with values for $p_0$ indicated by color. In each case, $p_1 = 1 - p_0$}
    \label{fig:unfair_example}
\end{figure} 

As any unfair coin produces a singular measure, all such unfair coin measures are singular with respect to Lebesgue measure (on $[0,1]$) and therefore singular with respect to the uniform measure on $[0,1]$.

In the sequel, we will extend this classic CDF result into a series of results on unfair $q$-sided dice.  This natural extension of the CDF formula to unfair dice is conjectured by Billingsley as an exercise and, to our knowledge no proof of this result exists in the literature. The following is a reinterpretation of this conjecture:

\begin{conjecture}[Problem 31.1 in \cite{billingsley}] \label{conj:billingsleydice}
    Let $p_0, \ldots, p_{q-1}$ be non-negative numbers adding to 1, where $q\geq 2$; suppose there is no $j$ such that $p_j=1$. Let $x_1, x_2,\ldots$ be independent, identically distributed random variables such that $P[x_i=j]=p_j$, $0\leq j < q$, and put $X = \sum_{i=1}^\infty x_i q^{-i}$. If $F(x)=\mathbb{P}[X\leq x]$ is the distribution function of $X$, then
    \begin{enumerate}[(i)]
        \item $F$ is continuous, \label{billingsleyi}
        \item $F$ is strictly increasing on $[0,1]$ if and only if $p_j>0$ for all $j$,\label{billingsleyii}
        \item if $p_j = q^{-1}$ for all $j$, then $F(x)=x$ on $[0,1]$; and \label{billingsleyiii}
        \item if $p_j\neq q^{-1}$ for some $j$, then $F$ is singular. \label{billingsleyiv}
    \end{enumerate}
\end{conjecture}

In addition to the preceding proposition on a single (unfair) die, recent research has referenced pairs of dice as well; notably, \cite{cornean2022characterization} and \cite{cornean2022singular}. While their work is focused on the generalized \textit{stationarity} setting, they provide a discussion of the problem's history in the i.i.d. case---a so-called `Bernoulli scheme'. They identify the $q=2$ CDF as a `Riesz-Nagy' function, and explicitly examine the Cantor function for $q=3$ (Problem 31.2 in \cite{billingsley}).

In \cite[\S1.1]{cornean2022characterization}, the authors go onto to make two claims without proof:
\begin{enumerate}[I.]
    \item The measures $\text{d}F=\mu$ in all the Bernoulli schemes for any $q$ are again all singular with respect to one another. \label{folklore1}
    \item Only one measure is absolutely continuous relative to Lebesgue measure; namely where all $j\in\{0,\ldots,q-1\}$ are equally likely. In this case, $\text{d}F=\lambda$ is Lebesgue measure itself on $[0,1]$.\label{folklore2}
\end{enumerate}
The first of which they refer to as a `folk theorem'. While they refer the reader to Section 14 of \cite{billingsleyergodic} for a discussion on the matter, we were unable to reproduce these observations from this text. That said, \cite[Example 3.5]{billingsleyergodic} does allude to the base-$q$ case for $q\geq 2$. Here, similar questions are posed as those \cite[Problems 31.1 and 32.1]{billingsley}, but still no proofs are given. 

In part, the next section will provide proofs to the aforementioned facts. Specifically, we prove Conjecture 
\ref{conj:billingsleydice} in Theorems \ref{maintheorem1} and \ref{maintheorem2}. We prove claims (\ref{folklore1}) and (\ref{folklore2}) in 
Theorems \ref{maintheorem2} and \ref{maintheorem3}. Afterward, in Theorem \ref{thm:noniidcoinlips}, we provide a novel analogous result to \cite[Example 31.1]{billingsley} for independent, but not identically distributed, coin flip sequences.

%% file: _d_main_dice.tex
In this section we begin by establishing qualitative results for the cumulative distribution functions associated with sequences obtained from unfair dice rolls. We then follow this discussion with the development of some machinery one can use to compare the unfair cumulative distribution function with the uniform distribution.

\subsection{Analysis of the CDF.}

Our first goal is prove parts (\ref{billingsleyi}) and (\ref{billingsleyii}) of Conjecture \ref{conj:billingsleydice}: 

\begin{theorem} \label{maintheorem1}
Consider a $q$-sided die, where $x_i\in\{0,\ldots,q-1\}$ denotes the outcome of the $i^\text{th}$ toss and $p_j:=\mathbb{P}[x_i=j]$ for $j\in\{0,\ldots,q-1\}$. Given $X = \sum_{i=1}^\infty x_iq^{-i}$, if $F(x)=\mathbb{P}[X\leq x]$ is the distribution function obtained after tossing the die an infinite number of times, then
\begin{enumerate}[(i)]
    \item $F$ is continuous;
    \item $F$ is strictly increasing on $[0,1]$ if and only if $p_j>0$ for all $j$;
    \item $F'$ exists almost everywhere in $[0,1]$;
\end{enumerate}
\end{theorem}
\begin{proof}
    For an arbitrary sequence $(u_1,u_2,\ldots)$ taken from $\{0,1,\ldots, q-1\}$, let $p_{u_i}:= \mathbb{P}[x_i=u_i]$. Since each $p_{u_i}<1$, we have
\begin{equation} \label{zerolimit}
    \mathbb{P}[x_i = u_i, \ i=1,2\ldots] = \lim_{n\to \infty} p_{u_1}\cdot p_{u_2}\cdot\ldots\cdot p_{u_n} = 0.
\end{equation}
Letting $x = \sum_{i=1}^\infty \frac{u_i}{q^i}$ be the (essentially) unique base-$q$ expansion for a number in $[0,1]$, we see immediately that $\mathbb{P}[X = x]=0$. Hence,
\[\mathbb{P}[X \leq x] = \mathbb{P}[X < x] + \mathbb{P}[X=x]=\mathbb{P}[X<x].\] 
It follows that $F$ is left-continuous. As a distribution, $F$ must be right-continuous. Therefore $F$ is everywhere continuous.

Now let $k\in \mathbb{N}$ so that $0\leq \frac{k}{q^n}\leq 1$. We can see that 
\[
\frac{k}{q^n} = \sum_{i=1}^n \frac{u_i}{q^i}
\]
for some $u_i\in \{0,1,\ldots,q-1\}$. Since $F$ is continuous, 
\begin{align}
\begin{split}
    F\left( \frac{k+1}{q^n}\right) - F\left( \frac{k}{q^n}\right) &= \mathbb{P}\left[X < \frac{k+1}{q^n}\right] - \mathbb{P}\left[X < \frac{k}{q^n}\right] \\
    &=\mathbb{P}\left[\frac{k}{q^n} < X < \frac{k+1}{q^n}\right]\\
    &=\mathbb{P}[x_i = u_i, \ i=1,2,\ldots,n] \\
    &= p_{u_1}\cdot\ldots\cdot p_{u_n}.
\end{split}
\label{eq:interval}
\end{align}
Therefore, since base-$q$ expansions are dense in $[0,1]$, $F$ is strictly increasing on $[0,1]$ if and only if $p_j>0$ for all $j$.  In any case, $F$ is non-decreasing and therefore, by Theorem 31.2 in \cite{billingsley}, $F'$ exists almost everywhere in $[0,1]$. 
\end{proof}

To proceed, we require two results on the frequency of digits in our base-$q$ expansions (both of which are due to \'Emile Borel). To discuss both, we fix the following notation: Given a finite set of $b$-digits, $B$, and an infinite sequence $\omega$ taken from $B$, let $\#_\omega(a,n)$ denote the number of times $a$ shows up in the first $n$ terms of $\omega$.

The first result we need is known as \textit{Borel's law of large numbers} (see \cite{2011aaf1-ab6b-3b14-bb2c-7b5bffa2d318} for an analytic proof) which states that if $S_n$, $n\geq 1$, is the number of successes in the first $n$ independent repetitions of a Bernoulli trial with success probability $p$, $0<p<1$, then
\[
\mathbb{P}\left( \lim_{n\to\infty} \frac{S_n}{n} =p \right) =1.
\]
In the context of our paper, our Bernoulli trial is the tossing of a $q$-sided die. Borel's law of large numbers is then asserting that, with probability 1, the frequency of each individual outcome tends toward its probability. From this, Lemma~\ref{lemma:blln} directly follows.

\begin{lemma} \label{lemma:blln}
    Consider a $q$-sided die, where $x_i\in\{0,\ldots,q-1\}$ denotes the outcome of the $i^\text{th}$ toss and $p_j:=\mathbb{P}[x_i=j]$ for $j\in\{0,\ldots,q-1\}$. Given $X = \sum_{i=1}^\infty x_iq^{-i}$, if $F(x)=\mathbb{P}[X\leq x]$ is the distribution function obtained after tossing the die an infinite number of times with $\mu$ as its associated probability measure, and $\omega_q = (d_1(x), d_2(x), \ldots)$ is the sequence of digits in the non-terminating base-$q$ expansion of an $x\in[0,1]$,  then
    \[
    \mu\left( \left\{ x\in (0,1]\colon \lim_{n\to\infty} \frac{\#_{\omega_q}(j,n)}{n} = p_j \right\} \right) = 1.
    \]
\end{lemma}

The second result we need is Borel's \textit{Normal Number Theorem}. While this result  was originally established  in \cite{mileBorelLesPD}, we refer the reader to \cite[Chapter 8]{kuipers2012uniform} for more details. By definition, given a finite set of $b$ digits, $B$, an infinite sequence $\omega$ on this set is \textbf{(simply) normal} if
 \[
 \lim_{n\to \infty} \frac{\#_\omega(a,n)}{n} = \frac{1}{b}
 \]
 for any $a\in B$. Thus, a sequence $\omega$ is normal for a set $B$ if the relative frequency of each item in $B$ is `fair'. The normal number theorem says that almost every real number $x$ is normal in any integral base $b>1$. Utilizing our notation, we formally write
 
 \begin{lemma}[Normal Number Theorem] \label{lemma:bnnt}
     Let $x\in [0,1]$, $b>1$ be an integer, and $B=\{0,\ldots,b-1\}$. If $\omega_b$ is the sequence of digits from $B$ that form the base-$b$ expansion of $x$, then
         \[
    \lambda\left( \left\{ x\in [0,1]\colon \lim_{n\to\infty} \frac{\#_{\omega_b}(j,n)}{n} = \frac{1}{q} \right\} \right) = 1
    \]
    for all $j\in B$, where $\lambda$ is Lebesgue measure on $[0,1]$.
 \end{lemma}
 Next, we record a measure-theoretic definition and proposition that have been localized to $[0,1]$. Both are reproduced directly from \cite[pg.410]{billingsley}.
 
\begin{definition} \label{def:disjointsupp}
 Two measures $\mu$ and $\lambda$ on $[0,1]$ have \textit{disjoint supports} if there exist Borel sets $S_\mu$ and $S_\lambda$ such that
 \[
 \mu([0,1]\setminus S_\mu) = 0, \hspace{.25in} \lambda([0,1]\setminus S_\lambda)=0, \hspace{.25in} \text{and} \hspace{.25in} S_\mu \cap S_\lambda=\emptyset.
 \]
\end{definition}

\begin{proposition} \label{prop:disjointsupports}
    If $F\colon [0,1]\to [0,1]$ is a differentiable function for which $\mu((a,b]) = F(b)-F(a)$, then $\mu$ and Lebesgue measure $\lambda$ have disjoint supports if and only if $F'(x)=0$ except on a set of Lebesgue measure 0.
\end{proposition}

We may now state and prove our second main result which proves parts (\ref{billingsleyiii}) and (\ref{billingsleyiv}) of Conjecture \ref{conj:billingsleydice} and simultaneously proves claim (\ref{folklore2}):
\begin{theorem}\label{maintheorem2}
Consider a $q$-sided die, where $x_i\in\{0,\ldots,q-1\}$ denotes the outcome of the $i^\text{th}$ toss and $p_j:=\mathbb{P}[x_i=j]$ for $j\in\{0,\ldots,q-1\}$. Given $X = \sum_{i=1}^\infty x_iq^{-i}$, if $F(x)=\mathbb{P}[X\leq x]$ is the cumulative distribution function obtained after tossing the die an infinite number of times, then
\begin{enumerate}[(i)]
    \item If $p_j= \textstyle \frac{1}{q}$ for all $j$, then $F(x)=x$ on $[0,1]$ and \label{maintheorem2item1}
    \item If $p_k\neq \textstyle \frac{1}{q}$ for some $k$, then $F$ is singular.\label{maintheorem2item2}
\end{enumerate}
In either case, $F$ is given by the following recursion formula:
    \begin{equation} \label{qsideddierecursion}
    F(x) = 
    \begin{cases} 
    p_0F(qx) &\text{if} \ 0\leq x \leq \textstyle \frac{1}{q}\\
    p_0+p_1F(qx-1) &\text{if} \ \textstyle \frac{1}{q} \leq x \leq \frac{2}{q}\\
    \ \ \ \ \ \ \ \ \vdots &\ \ \ \ \ \ \ \ \ \vdots\\
    (p_0+p_1+\ldots+p_{q-2}) + p_{q-1}F(qx-(q-1)) &\text{if} \ \textstyle \frac{q-1}{q} \leq x \leq 1\\
    \end{cases}.
    \end{equation}
\end{theorem}
\begin{proof}
While (\ref{maintheorem2item1}) will follow from the recursion formula established independently, we can also use the setting detailed in \eqref{eq:interval} to observe that, if $p_j=\textstyle \frac{1}{q}$ for all $j$, then
\[
F\left( \frac{k}{q^n}+\frac{1}{q^n}\right) - F\left( \frac{k}{q^n}\right) = F\left( \frac{k+1}{q^n}\right) - F\left( \frac{k}{q^n}\right) = p_{u_1}\cdot\ldots\cdot p_{u_n} = \frac{1}{q^n}.
\]
Hence, due to the density of base-$q$ expansions in $[0,1]$, $F(x)=x$.

We now establish (\ref{maintheorem2item2}). Take $x\in(0,1]$ and let $\omega_q = (d_1(x), d_2(x),\ldots)$ be the sequence of digits in its non-terminating base-$q$ expansion. If $\mu$ represents our probability measure, then $\mu[x\colon d_i(x)=j]=p_j$ for $j\in\{0,\ldots,q-1\}$. For every $j$, form
\[
S_j:= \left\{ x\in (0,1]\colon \lim_{n\to\infty} \frac{\#_{\omega_q}(j,n)}{n} = p_j \right\} \hspace{.25in} \text{and consider} \hspace{.25in} \widetilde{S}:=\bigcap_{j=0}^{q-1} S_j.
\]
Lemma \ref{lemma:blln} asserts that $\mu(S_j)=1$ for every $j$. Thus, by subadditivity of the measure $\mu$,
\[
\mu((\widetilde{S})^c) = \mu\left( \bigcup_{j=0}^{q-1} S_j^c \right) \leq \sum_{j=0}^{q-1} \mu(S_j^c) = 0,
\]
and therefore $\mu(\widetilde{S})=1$. Similarly, for every $j$ form
\[
T_j:= \left\{ x\in (0,1]\colon \lim_{n\to\infty} \frac{\#_{\omega_q}(j,n)}{n} = \frac{1}{q} \right\} \hspace{.25in} \text{and consider} \hspace{.25in} \widetilde{T}:=\bigcap_{j=0}^{q-1} T_j.
\]
Lemma \ref{lemma:bnnt} asserts that $\lambda(T_j)=1$ for every $j$, where $\lambda$ is Lebesgue measure. Thus, again by the subadditivity of $\lambda$,  $\lambda(\widetilde{T})=1$.

Now suppose that $p_k\neq \frac{1}{q}$ for some $k$. Then, by the uniqueness of limits, $S_k\cap T_k = \emptyset$ and therefore $\widetilde{S}\cap \widetilde{T}=\emptyset$. By Definition \ref{def:disjointsupp}, $\mu$ and $\lambda$ are seen to have disjoint supports. It now follows from Proposition \ref{prop:disjointsupports} that $F'(x)=0$ except on a set of Lebesgue measure 0 and therefore $F$ is singular.

Finally, we establish the recursion formula given in (\ref{qsideddierecursion}). Note that $[0,1]$ can be divided into $q$ intervals $[0, \textstyle \frac{1}{q}], [\textstyle \frac{1}{q}, \frac{2}{q}],\ldots, [\textstyle \frac{q-1}{q},1]$ --- the so-called base-$q$ intervals of rank 1. All cases in the recursion proceed in an identical fashion. As such, we provide an explicit proof of the last case of the recursion formula only. 

Suppose $x\in [\textstyle \frac{q-1}{q},1]$, the $q^\text{th}$ base-$q$ interval of rank 1. Here, $X \leq x$ can occur in $q$ different ways. Specifically, either $X$ lies in one of the previous base-$q$ intervals, or it lies in the last interval with $x$. Thus,
\begin{align*}
    \mathbb{P}[X \leq x] &= \mathbb{P}\left[ x_1=0 \ \text{or} \  \ldots \ \text{or} \ x_{1}=q-2 \ \text{or} \  \left(x_1=q-1 \ \text{and} \ \frac{q-1}{q} + \sum_{i=2}^\infty \frac{x_i}{q^i} \leq x\right)\right]\\
    &= (p_0 + \ldots + p_{q-2}) + p_{q-1}\mathbb{P}\left[q-1 + \sum_{i=2}^\infty \frac{x_i}{q^{i-1}} \leq qx\right]\\
    &= (p_0 + \ldots + p_{q-2}) + p_{q-1}\mathbb{P}\left[\sum_{i=1}^\infty \frac{x_{i+1}}{q^{i}} \leq qx-(q-1)\right]\\
    &= (p_0 + \ldots + p_{q-2}) + p_{q-1}\mathbb{P}\left[X \leq qx-(q-1)\right]\\
    &= (p_0 + \ldots + p_{q-2}) + p_{q-1}F(qx-(q-1)).
\end{align*}
Therefore, when $x\in [\textstyle \frac{q-1}{q},1]$, we have the recursion 
\[
F(x) = (p_0 + \ldots + p_{q-2}) + p_{q-1}F(qx-(q-1)).
\]
The rest of the cases follow similarly.
\end{proof}

\begin{remark} \label{cantor}
Note that the \textit{Cantor distribution} is the probability distribution whose cumulative distribution function is the Cantor function. This distribution is often given as an example of a singular distribution. As a result, it is worth noting that the singular distribution obtained in Theorem \ref{maintheorem2} is a generalization of the Cantor distribution. Indeed, if we let $q=3$, $p_0=p_2=0.5$, and $p_1=0$, Theorem \ref{maintheorem2} states that the resulting cumulative distribution function, $F(x)$, is singular and is given by the following recursion formula:
\[
F(x) = \begin{cases}
\frac{1}{2}F(3x) &\text{if} \ 0\leq x \leq \frac{1}{3}\\
\frac{1}{2}&\text{if} \ \frac{1}{3}\leq x \leq \frac{2}{3}\\
\frac{1}{2} + \frac{1}{2}F(3x-2) &\text{if} \ \frac{2}{3}\leq x \leq 1\\
\end{cases}
\]
By comparing this formula to that given in \cite{Dobos} and \cite[pg.9]{cantor1}, we see that this formula exactly defines the Cantor distribution whose graph is the `Devil's Staircase'.
\end{remark}

Notably, the proof given in Theorem \ref{maintheorem2} can be modified to form a stronger conclusion on singularity. Specifically, we can compare the probability measures obtained from differently weighted dice. The following result proves claim (\ref{folklore1}):

\begin{theorem}\label{maintheorem3}
Consider two $q$-sided dice with sides taken from $Q:=\{0, \ldots, q-1\}$. Let $x_i\in Q$ denote the outcome of the $i^\text{th}$ toss of one die with corresponding probabilities $p_j:=\mathbb{P}[x_i=j]$ for $j\in Q$, and let $\widetilde{x}_k$ and $\widetilde{p}_k$ be defined similarly for the second die (with $k\in Q)$. Given $X = \sum_{i=1}^\infty x_iq^{-i}$ and $\widetilde{X} = \sum_{i=1}^\infty \widetilde{x}_iq^{-i}$, put $F(x)=\mathbb{P}[X\leq x]$ and $\widetilde{F}(x)=\mathbb{P}[\widetilde{X}\leq x]$ as the respective cumulative distribution functions obtained after tossing the corresponding die an infinite number of times. If there exists an outcome $t\in Q$ so that $p_t \neq \widetilde{p}_t$, then the associated probability measures, $\mu$ and $\widetilde{\mu}$, are mutually singular. \label{twodice1}
\end{theorem}
\begin{proof}
    It suffices to show that $\mu$ and $\widetilde{\mu}$ have disjoint supports. We will essentially use the argument given in the proof of Theorem \ref{maintheorem2}, but using two probability measures instead of one probability measure and Lebesgue measure.
    
    Take $x\in(0,1]$ and let $\omega_q = (d_1(x), d_2(x),\ldots)$ be the sequence of digits in its non-terminating base-$q$ (equivalently, base-$\widetilde{q}$) expansion. If $\mu$ and $\widetilde{\mu}$ represent our two probability measures, then $\mu[x\colon d_i(x)=j]=p_j$ and  $\widetilde{\mu}[x\colon d_i(x)=k]=\widetilde{p}_k$ for $j,k\in Q$.
    
    For every $j$, form
\[
S_j:= \left\{ x\in (0,1]\colon \lim_{n\to\infty} \frac{\#_{\omega_q}(j,n)}{n} = p_j \right\} \hspace{.25in} \text{and} \hspace{.25in} S:=\bigcap_{j\in Q} S_j.
\]
Lemma \ref{lemma:blln} asserts that $\mu(S_j)=1$ for every $j$. Thus, by subadditivity of $\mu$, we have $\mu(\widetilde{S})=1$. Similarly, for every $k$, form
\[
\widetilde{S_k}:= \left\{ x\in (0,1]\colon \lim_{n\to\infty} \frac{\#_{\omega_{q}}(k,n)}{n} = \widetilde{p_k} \right\} \hspace{.25in} \text{and} \hspace{.25in} \widetilde{S}:=\bigcap_{k\in Q} \widetilde{S_k}.
\]
Again, by Lemma \ref{lemma:blln} and subadditivity of $\widetilde{\mu}$, we have $\widetilde{\mu}(\widetilde{S})=1$.
    
By assumption, there exists an outcome $t$ for which $p_t\neq \widetilde{p}_t$. Thus, by the uniqueness of limits, we have $S_t\cap \widetilde{S_t} = \emptyset$ and therefore $S\cap \widetilde{S} = \emptyset$. Hence, $\mu$ and $\widetilde{\mu}$ have disjoint supports.
\end{proof}

We have so far only addressed i.i.d. random dice rolls and coin flips. For the coin flips case, we can say a little bit more in the independent, but not necessarily identically distributed case.

\begin{theorem} \label{thm:noniidcoinlips}
Suppose we flip an infinite number of 2-sided unfair coins, but each have a different weighting. Specifically, let $x_i\in\{0,1\}$ denote the outcome of the $i^\text{th}$ flip and suppose $0<\mathbb{P}[x_i=0]=:p_{i;0}\neq 0.5$. If $(p_{i;0})\not\to 0.5$, then $F'(x)=0$ almost everywhere and therefore $F$ is singular.
\end{theorem}
\begin{proof} 
Analogous arguments to those in Theorem \ref{maintheorem1} demonstrate that $F$ is well-defined, continuous, and increasing, and therefore $F'$ exists almost everywhere in $[0,1]$. 

Suppose that $(p_{i;0})\not\to0.5$. We will demonstrate that $F'(x)=0$. Let $k\in \mathbb{N}$ so that $0\leq \frac{k}{2^n}\leq 1$. Then, $\textstyle \frac{k}{2^n} = \sum_{i=1}^n \frac{u_i}{2^i}$
for some $u_i\in \{0,1\}$ and
\begin{align*}
    F\left( \frac{k+1}{2^n}\right) - F\left( \frac{k}{2^n}\right) &= \mathbb{P}\left[X < \frac{k+1}{2^n}\right] - \mathbb{P}\left[X < \frac{k}{2^n}\right]\\
    &=\mathbb{P}\left[\frac{k}{2^n} < X < \frac{k+1}{2^n}\right]\\
    &=\mathbb{P}[x_i = u_i, \ i=1,2,\ldots,n]\\
    &= p_{u_1}\cdot\ldots\cdot p_{u_n}.
\end{align*}

Let $x$ be given and for each $m\in\mathbb{N}$, choose $k_m$ so that $x \in I_n$, where
\[
I_m = \left( \frac{k_m}{2^m}, \frac{k_m+1}{2^m}\right)
\]
is the dyadic interval of rank $m$ that contains $x$. It follows from the density of dyadics in $[0,1]$, that
\[
    \lim_{m\to\infty} \frac{ \mathbb{P}[X \in I_m]}{2^m} = \lim_{m\to\infty} \frac{ F\left( \frac{k+1}{2^m}\right) - F\left( \frac{k}{2^m}\right)   }{2^m} = F'(x).
\]
Therefore, if we suppose, for the sake of contradiction, that $F'(x)\neq 0$, then on one hand we obtain the following:
\[
    \lim_{m\to\infty} \frac{\frac{ \mathbb{P}[X \in I_{m+1}]}{2^{m+1}}}{ \frac{ \mathbb{P}[X \in I_{m}]}{2^m} } = \frac{F'(x)}{F'(x)}=1.
\]
Thus, 
\begin{equation} \label{1/qcont}
    \lim_{m\to\infty} \frac{ \mathbb{P}[X \in I_{m+1}]}{ \mathbb{P}[X\in I_{m}] } = \frac{1}{2}.
\end{equation}
On the other hand, We know that $I_m$ consists of those numbers in $[0,1]$ whose dyadic expansions match $x$'s for the first $m$ terms. Thus, if $X \in I_m$, then 
\[
X = \sum_{i=1}^m \frac{u_i}{q^i} + \sum_{i=m+1}^\infty \frac{x_i}{q^i}.
\]
This implies that $\mathbb{P}[X \in I_m] = p_{u_1}\cdot p_{u_2}\cdot\ldots \cdot p_{u_m}$. Therefore,
\[
    \frac{ \mathbb{P}[X \in I_{m+1}]}{ \mathbb{P}[X \in I_{m}] } = \frac{  p_{u_1}\cdot p_{u_2}\cdot\ldots \cdot p_{u_m} \cdot p_{u_{m+1}}}{p_{u_1}\cdot p_{u_2}\cdot\ldots \cdot p_{u_m}} = p_{ u_{m+1} }.
\]
We assumed that $(p_{i;0})\not\to 0.5$, therefore
\[
    \lim_{m\to\infty}\frac{ \mathbb{P}[X \in I_{m+1}]}{ \mathbb{P}[X \in I_{m}] } = \lim_{m\to\infty} p_{ u_{m+1} }\neq \frac{1}{2}.
\]
This contradicts the conclusion yielded in (\ref{1/qcont}). Thus, $F'(x) = 0$ almost everywhere and therefore $F$ is singular.
\end{proof}

%% file: _e_comparisons.tex
If computational devices can be used to create distributions based on possibly unfair coin tosses or die rolls, it is natural to ask how far away results could be expected to be from uniform.  In practice, such comparisons are likely to be statistical in nature.  However, given the results of the previous sections, we now have firm distributional objects with which to compare. In this section we offer two analytic ways to compare an unfair distribution to the uniform and fair one.  The first is done in the infinite toss limit and utilizes the sup-norm.  The second considers the practicality of the finite and compares a finite number of rolls or tosses to uniform through arclength.

\subsection{Comparison to Uniform under $\|\cdot\|_\infty$} \label{supcompare}
In this section we establish a method to compare the (possibly unfair) distribution $F(x)$ with the uniform (fair) distribution using the sup-norm. To start, consider the operator $T\colon C[0,1]\to C[0,1]$ defined by
\begin{equation} \label{seqdef}
(Tf)(x) =     \begin{cases} 
    p_0f(qx) &\text{if} \ 0\leq x \leq \textstyle \frac{1}{q}\\
    p_0+p_1f(qx-1) &\text{if} \ \textstyle \frac{1}{q} \leq x \leq \frac{2}{q}\\
    \ \ \ \ \ \ \ \ \vdots &\ \ \ \ \ \ \ \ \ \vdots\\
    (p_0+p_1+\ldots+p_{q-2}) + p_{q-1}f(qx-(q-1)) &\text{if} \ \textstyle \frac{q-1}{q} \leq x \leq 1\\
    \end{cases}
\end{equation}
\begin{lemma} \label{lemma:contraction}
    Put $p_{\text{max}} = \max\{p_i\}_{i=1}^q$. If $f,g\in C[0,1]$, then $\|Tf-Tg\|_\infty \leq p_{\text{max}}\|f-g\|_\infty$, and therefore $T$ defines a contraction mapping.
\end{lemma}
\begin{proof}
    Suppose $\frac{q-1}{q} \leq x \leq 1$. Then,
    \begin{align*}
        Tf-Tg &= \left(\sum_{j=0}^{q-2}p_j + p_{q-1}f(qx-(q-1))\right) - \left(\sum_{j=0}^{q-2}p_j + p_{q-1}g(qx-(q-1))\right)\\
        &= p_{q-1}(f(qx-(q-1)) - g(qx-(q-1))).
    \end{align*}
    Therefore $\|Tf-Tg\|_\infty = p_{q-1}\|f-g\|_\infty$ when restricted to $\left[\frac{q-1}{q},1\right]$. Similar conclusions hold for the other subsets of the domain. Putting them all together (and supping over all of $[0,1]$), we conclude
    \[
    \|Tf-Tg\|_\infty \leq p_{\text{max}}\|f-g\|_\infty.
    \]
\end{proof}
\begin{theorem} \label{theorem: uniformcomp}
    Given the sequence of functions $(f_n)_{n=0}^\infty$ defined by $f_0=x$ and $f_{n+1}= Tf_n$ and the distribution function $F$ in Theorem \ref{maintheorem2},
    \begin{enumerate}[(i)]
        \item $f_n \to F$
        \item $\|x-F(x)\|_\infty \leq \displaystyle\left( \frac{1}{1-p_{\text{max}}}\right)\|x-f_1(x)\|_\infty$
    \end{enumerate}
\end{theorem}
\begin{proof}
    Lemma \ref{lemma:contraction} showed that the operator $T\colon C[0,1]\to C[0,1]$ is a contraction mapping. Thus, since $(C[0,1],\|\cdot\|_\infty)$ is a (non-empty) complete metric space, the Banach Fixed Point Theorem guarantees that $T$ admits a unique fixed point---a function $G\in C[0,1]$ such that $TG=G$, where $G=\lim f_n$. This function is exactly the distribution $F$ in Theorem \ref{maintheorem2}.
    
    The Banach Fixed Point Theorem implies that you'll get the same fixed point no matter what starting function you pick (i.e., the choice of $f_0\in C[0,1]$ is arbitrary). As such, we can choose our initial starting function as $f_0=x$ and denote $F(x) =: f_\infty(x)$. Now, since $x-F(x) = f_0-f_\infty$ is seen to be a telescoping series, ,we have the following:
   \[
\|x-F(x)\| = \|f_0 - f_\infty\| = \left\| \sum_{n=0}^\infty f_n-f_{n+1} \right\| \leq \sum_{n=0}^\infty \|f_n-f_{n+1}\|.
\]
Which, by Lemma \ref{lemma:contraction} and the fact that $|p_{\text{max}}|<1$, yields
\[
\|x-F(x)\|  \leq \sum_{n=0}^\infty (p_{\text{max}})^n\|f_0-f_1\| = \displaystyle\left( \frac{1}{1-p_{\text{max}}}\right)\|f_0-f_1\| = \displaystyle\left( \frac{1}{1-p_{\text{max}}}\right)\|x-f_1\|,
\]
where all norms are sup-norms over $[0,1]$.
\end{proof}

Thus, to understand the sup-norm difference between the distribution $F(x)$ and the uniform distribution $y=x$, it suffices to understand the quantity $\|x-f_1\|$, where $f_1$ is a (finite) piece-wise linear function and has no fractal-like components.

\subsection{Comparisons via Arclength} \label{comparearclength}
The previous section's result gave comparative information about the full singular distribution---one obtained after we roll our unfair die infinitely often. What if we wanted to compare our (possibly unfair) distribution after finitely many rolls? One route is to look at their arclengths. We start by recording a small result from \cite[Theorem 6.22]{cantor1}:

\begin{lemma} \label{lemma:arclength}
    Let $F\colon [0,1]\to \mathbb{R}$ be a continuous, increasing function for which $F(0)=0$ and $F(1)=1$. Then the following two statements are equivalent
    \begin{enumerate}[(i)]
        \item The length of the arc $y=F(x)$ on $[0,1]$ is $2$.
        \item The function $F$ is singular.
    \end{enumerate}
\end{lemma}

\begin{proposition} \label{prop:arclengthbounds}
  Consider a $q$-sided die, where $x_i\in\{0,\ldots,q-1\}$ denotes the outcome of the $i^\text{th}$ toss and $p_j:=\mathbb{P}[x_i=j]$ for $j\in\{0,\ldots,q-1\}$. Given $X = \sum_{i=1}^\infty x_iq^{-i}$, if $F(x)=\mathbb{P}[X\leq x]$ is the cumulative distribution function obtained after tossing the die an infinite number of times, then
  \begin{enumerate}[(i)]
      \item If $p_j = \frac{1}{q}$ for all $j$, then the arclength of $F(x)$ on $[0,1]$ is $\sqrt{2}$. \label{item1:prop:arclengthbounds}
      \item If $p_j\neq \frac{1}{q}$ for some $j$, then the arclength of $F(x)$ on $[0,1]$ is $2$.
  \end{enumerate}
\end{proposition}
\begin{proof}
    By Theorem \ref{maintheorem2}, if $p_j = \frac{1}{q}$ for all $j$, then $F(x)=x$ on $[0,1]$ and therefore its length is $\sqrt{2}$. If, on the other hand, we have that $p_j\neq \frac{1}{q}$ for some $j$, then Theorem \ref{maintheorem2} guarantees that $F(x)$ is singular. Moreover, by Theorem \ref{maintheorem1}, we know that $F$ is both continuous and increasing on $[0,1]$. Finally, we observe that $F(0)=0$ and $F(1)=1$ (this can be seen, for example, by using the recursion formula in Theorem \ref{maintheorem2}). Thus, by Lemma \ref{lemma:arclength}, the arclength of $F$ on $[0,1]$ is equal to 2.
\end{proof}

This theorem tells us the arclength for the cumulative distribution function after we roll our die infinitely often. If, however, we are given a die and roll it $n$ times, we can ask: Are we getting closer to a fair distribution, or an unfair one? One way to answer this is to look at the $n^{\text{th}}$ iterate of our recursion formula. That is, in the language of Theorem \ref{theorem: uniformcomp}, we consider $f_n = T^nx$ and look at its arclength as $n$ gets larger.

We fix some notation first. Given our $q$-sided die, where $x_i\in\{0,\ldots,q-1\}$ denotes the outcome of the $i^\text{th}$ toss and $p_j:=\mathbb{P}[x_i=j]$ for $j\in\{0,\ldots,q-1\}$, we will put $P = \{p_j\}_{j=0}^{q-1}$ and denote the set of its $n$-tuples by $P^n$. Order this finite set lexicographically and put $P^n =\{v_\ell\}_{\ell=0}^{q^n-1}$ so that, for example, $v_0 = (p_0,p_0,\ldots,p_0)$, $v_1=(p_0,p_1,p_0,\ldots,p_0)$, etc. Let $\Pi_n\colon P^n\to \mathbb{R}$ be the mapping that multiplies the coordinates of a given tuple from $P^n$. For example, if $v_\ell = (p_0,p_1,p_0,p_2,\ldots,p_3)$, then $\Pi_n(v_\ell) = p_0\cdot p_1\cdot p_0\cdot p_2\cdot \ldots \cdot p_3$.

\begin{remark} \label{remark:tags}
    The  tuple  of  probabilities associated with a  given  $v_\ell$,  say, $(p_{\ell_0},\ldots,p_{\ell_{q-1}})$ provide a \textit{unique} `tag' by which we can locate the $q^n$ base-$q$ intervals of rank $n$. For example, if we let $q=4$ so that our probabilities are $p_0,p_1,p_2$, and $p_3$, and we consider the base-$4$ intervals of rank $n=2$, then 
    \[
    P^2 = \{(p_i,p_j) \ : \ i,j\in\{0,1,2,3\}\}.
    \]
    Now, the interval $\left[ \frac{14}{4^2}, \frac{15}{4^2}\right]$ is uniquely associated with the tuple $v_{14}=(p_3,p_2)$ in the following way: Start with $[0,1]$, zoom in on the fourth subinterval $\left[ \frac{3}{4}, 1\right]$ and further zoom in on \textit{its} third subinterval to yield $\left[ \frac{14}{4^2}, \frac{15}{4^2}\right]$. Note that order matters. For example, the tuple $v_7=(p_2,p_3)$ corresponds to the interval $\left[ \frac{7}{4^2}, \frac{8}{4^2}\right]$. In either case, $\Pi_2((p_3,p_2))=\Pi_2((p_2,p_3)) = p_3p_2$.
\end{remark}

\begin{theorem} \label{theorem:arclength}
     Consider a $q$-sided die, where $x_i\in\{0,\ldots,q-1\}$ denotes the outcome of the $i^\text{th}$ toss and $p_j:=\mathbb{P}[x_i=j]$ for $j\in\{0,\ldots,q-1\}$. Let $f_n = T^nx$ be the piece-wise linear function described in Theorem \ref{theorem: uniformcomp} and let $P^n =\{v_i\}_{i=0}^{q^n-1}$ be the set of $q$-tuples described in the preceding conversation. Then
     \[
        \text{Arclength of} \ f_n = \sum_{i=0}^{q^n-1} \sqrt{ \left( \frac{1}{q^n} \right)^2 +  (\Pi_n(v_{i}))^2 }.
     \]
\end{theorem}
\begin{proof}
    The function $f_n$ is piece-wise linear on the base-$q$ intervals of rank $n$: 
 \begin{equation} \label{eq:ranknints}
         \left[0, \frac{1}{q^n}\right], \left[\frac{1}{q^n}, \frac{2}{q^n}\right],\ldots, \left[ \frac{q^n-1}{q^n},1\right].
 \end{equation}
    As such, the arclength of $f_n$ is equal to the sum of the length of the linear components on each of these intervals. Let $[\textstyle \frac{i}{q^n}, \frac{i+1}{q^n}]$ be an arbitrary such interval and consider the points $f_n(\frac{i}{q^n})$ and $f_n(\frac{i+1}{q^n})$.

    Note that, if $F$ is the full cumulative distribution function, then $F$ and $f_n$ agree on the endpoints of every base-$q$ interval of rank $n$. (In fact, this is true for any rank of base-$q$ intervals.) Thus, we can use the recursive definition for $F$ given in Theorem \ref{maintheorem2} to evaluate the endpoints $f_n(\frac{i}{q^n})$ and $f_n(\frac{i+1}{q^n})$. To see this, note that the points $\frac{i}{q^n}$ and $\frac{i+1}{q^n}$ must both live in some base-$q$ interval of rank $1$. That is, they must both live in one of $[0, \textstyle \frac{1}{q}], [\textstyle \frac{1}{q}, \frac{2}{q}],\ldots, [\textstyle \frac{q-1}{q},1]$. (Here we are allowing the possibility that one of these points is the endpoint of an interval). Suppose the two points live in the interval $[\textstyle \frac{k}{q}, \frac{k+1}{q}]$ for some $k$. Then
\begin{align*}
   f_n\left(\frac{i}{q^n}\right) = F\left(\frac{i}{q^n}\right) &= (p_0+p_1+\ldots+p_{k-1})+p_{k}F\left(q\cdot \frac{i}{q^n} - k \right) \\
   &= (p_0+p_1+\ldots+p_{k-1})+p_{k}F\left(\frac{i - kq^{n-1}}{q^{n-1}}\right)
\end{align*}
and
\begin{align*}
   f_n\left(\frac{i+1}{q^n}\right) = F\left(\frac{i+1}{q^n}\right) &= (p_0+p_1+\ldots+p_{k-1})+p_{k}F\left(q\cdot \frac{i+1}{q^n} - k \right) \\
   &=(p_0+p_1+\ldots+p_{k-1})+p_{k}F\left(\frac{(i - kq^{n-1})+1}{q^{n-1}}\right)
\end{align*}
so that 
   \begin{equation} \label{eq:functionheight}
f_n\left(\frac{i+1}{q^n}\right) - f_n\left(\frac{i}{q^n}\right) = p_k\left( F\left(\frac{(i - kq^{n-1})+1}{q^{n-1}}\right) - F\left(\frac{i - kq^{n-1}}{q^{n-1}}\right) \right)
    \end{equation}
Notably, we are now looking at the endpoints of the base-$q$ interval of rank $n-1$ given by $\left[ \textstyle \frac{i - kq^{n-1}}{q^{n-1}}, \frac{(i - kq^{n-1})+1}{q^{n-1}} \right]$. Thus, (\ref{eq:functionheight}) shows that when we run the difference $f_n\left(\frac{i+1}{q^n}\right) - f_n\left(\frac{i}{q^n}\right)$ through an iteration of the recursive formula for $F$, it returns the probability $p_k$ (where $k$ is the rank-1 interval that the rank-$n$ interval $\left[ \textstyle \frac{i}{q^n},\frac{i+1}{q^n}\right]$ landed in) times another difference of $F$-function evaluations at the endpoints of a rank $n-1$ interval. 

Repeating this process $n$ times, we get that
\begin{equation*}
    f_n\left(\frac{i+1}{q^n}\right) - f_n\left(\frac{i}{q^n}\right) = \Pi_n(v_{i})(F(1)-F(0)) = \Pi_n(v_{i})(1-0) = \Pi_n(v_{i}),
\end{equation*}
where, in light of Remark \ref{remark:tags}, $v_{i}$ is exactly the unique $q$-tuple `tag' for the interval $\left[ \textstyle \frac{i}{q^n},\frac{i+1}{q^n}\right]$.
    
    Since $f_n$ is piece-wise linear on the base-$q$ intervals of rank $n$, its arclength over $\left[ \textstyle \frac{i}{q^n},\frac{i+1}{q^n}\right]$ is equal to
    \[
\sqrt{\left( \frac{i+1}{q^n} - \frac{i}{q^n} \right)^2 + \left(f_n\left(\frac{i+1}{q^n}\right) - f_n\left(\frac{i}{q^n}\right)\right)^2} = \sqrt{\left( \frac{1}{q^n}\right)^2 + \left(\Pi_n(v_{i})\right)^2}
    \]
    Computing this for each of interval in (\ref{eq:ranknints}) gives the total arclength as
         \[
        \text{Arclength of} \ f_n = \sum_{i=0}^{q^n-1} \sqrt{ \left( \frac{1}{q^n} \right)^2 +  (\Pi_n(v_{i}))^2 }.
     \]
\end{proof}

Observe that, indeed, if $p_j=\textstyle\frac{1}{q}$ for all $j$, then $\Pi_n(v_i)=\textstyle \frac{1}{q^n}$ for all $v_i\in P^n$ so that
\begin{align*}
\text{Arclength of} \ F \ \text{on} \ [0,1]  &= \lim_{n\to \infty} \sum_{i=0}^{q^n-1} \sqrt{ \left( \frac{1}{q^n} \right)^2 +  (\Pi_n(v_{i}))^2 }\\
&= \lim_{n\to \infty} \sum_{i=0}^{q^n-1} \sqrt{ \left( \frac{1}{q^n} \right)^2 +  \left( \frac{1}{q^n} \right)^2 }\\
&= \sqrt{2},
\end{align*}
Thus, this formula now provides an alternate way to recover (\ref{item1:prop:arclengthbounds}) from Proposition \ref{prop:arclengthbounds}. We also have the following corollary:
\begin{corollary} \label{corr:unfairarc}
      If $p_j\neq \textstyle \frac{1}{q}$ for some $j$, then 
      \[
\text{Arclength of} \ F \ \text{on} \ [0,1] = \lim_{n\to\infty} \sum_{i=0}^{q^n-1} \sqrt{ \left( \frac{1}{q^n} \right)^2 +  (\Pi_n(v_{i}))^2 } = 2.
      \]
      \end{corollary}
      \begin{proof}
          The fact that $p_j\neq \textstyle \frac{1}{q}$ implies, by Proposition \ref{prop:arclengthbounds}, that the arclength of $F$ is 2. Using the formula in Theorem \ref{theorem:arclength} gives the result.
      \end{proof}
      In general, the authors know of no way to compute this limit directly. This might be interesting due to the fact that result can be rephrased as 
         \[
\lim_{n\to\infty} \left(\underset{v_i\in P^n}{\text{Average}}\left\{\sqrt{ 1 +  (q^n\Pi_n(v_i))^2 }\right\}\right) = 2,
 \]
 where we van view $\log(p_i)$ as being the weights on a complete $q$-ary tree and the mapping $\Pi_n$ as computing length of the paths (in a graph-theoretic sense) through the tree. The expression above is then taking a limit of the average distance, over all paths in the tree, between 1 and the perturbation-from-fair that a particular path yields.

 \begin{remark}
     In light of Remark \ref{cantor}, we know that the Cantor function is obtained when $q=3$ and our probabilities are chosen so that $p_0=p_2=0.5$ and $p_1=0$. Proposition \ref{prop:arclengthbounds} therefore guarantees that its arclength on $[0,1]$ is equal to 2 and hence, we know that the limit considered in Corollary \ref{corr:unfairarc} is equal to 2. Interestingly, this is one of the few instances that one \textit{can} compute this limit directly. See \cite[pg.4]{cantor2} for details.
 \end{remark}

%% file: _f_discussion.tex
Motivated by the need to understand probabilistic computing devices and their inherent randomness, our paper aimed to investigate the distribution function associated with rolling a (possibly unfair) $q$-sided die. While the literature covers the $q=2$ coin flip case extensively, the full $q$-sided die case had not yet been addressed. Theorems \ref{maintheorem1} and \ref{maintheorem2} helped answer Conjecture  \ref{conj:billingsleydice}, while Theorems \ref{maintheorem2} and \ref{maintheorem3} addressed recent `folk lore' claims (reproduced here as claims (\ref{folklore1}) and (\ref{folklore2}).) In the spirit of these results, Theorem \ref{thm:noniidcoinlips} provided a novel analogous result to \cite[Example 31.1]{billingsley} for independent, but not identically distributed, coin flip sequences.

Adding to these investigations, we have provided two theoretical tools to compare an unfair distribution, $F(x)$, to a fair one, $y=x$. Both use the iterative construction of the distribution given in (\ref{seqdef}). Theorem \ref{theorem: uniformcomp} in Section \ref{supcompare} provided an upper bound on $\|x-F(x)\|_\infty$, and Theorem \ref{theorem:arclength} in Section \ref{comparearclength} provided a formula for calculating the arclength of the $F(x)$ after finitely many dice rolls.

\subsection{Future Mathematical Work}
In this paper, we investigated the analytic properties of distributions associated with unfair dice. We looked at a single die in Theorems \ref{maintheorem1} and \ref{maintheorem2} and a pair of same-sided dice in Theorem \ref{maintheorem3}. It is reasonable to ask about analytic comparisons between pairs of dice with differing number of sides. 

Consider a $q$-sided and $\widetilde{q}$-sided pair of dice. Let $x_i\in\{0,\ldots,q-1\}$ denote the outcome of the $i^\text{th}$ toss of the $q$-sided die with corresponding probabilities $p_j:=\mathbb{P}[x_i=j]$ for $j\in\{0,\ldots,q-1\}$, and let $\widetilde{x}_k$ and $\widetilde{p}_k$ be defined similarly for the second die (with $k\in\{0,\ldots,\widetilde{q}-1\})$. Given $X = \sum_{i=1}^\infty x_iq^{-i}$ and $\widetilde{X} = \sum_{i=1}^\infty \widetilde{x}_i\widetilde{q}^{-i}$, put $F(x)=\mathbb{P}[X\leq x]$ and $\widetilde{F}(x)=\mathbb{P}[\widetilde{X}\leq x]$ as the respective cumulative distribution functions obtained after tossing the corresponding die an infinite number of times. Let $\mu$ and $\widetilde{\mu}$ be the associated probability measures, respectively.

\begin{question}
 If $q\neq \widetilde{q}$, then what can be said of $\mu$ and $\widetilde{\mu}$?   
\end{question}

As far as we can tell, this situation is more nuanced. It seems possible to show that when $q>\widetilde{q}$, and the $q$ die has only $\widetilde{q}$ possible outcomes (with the $q-\widetilde{q}$ outcomes having zero probability), the associated measures are mutually singular. However, when $q$ and $\widetilde{q}$ are not relatively prime (for example, when $q=2$ and $\widetilde{q}=4$), it may be the case that, for some specific choices of probabilities, the support of $\mu$ and $\widetilde{\mu}$ could be the same. It is further unclear what happens when $q$ and $\widetilde{q}$ \textit{are} taken to be relatively prime. Tackling this generalization would require a more thorough understanding of how the measures $\mu$ and $\widetilde{\mu}$ interact with base expansions different from their own.

In Corollary \ref{corr:unfairarc}, we showed that when $p_j\neq \textstyle \frac{1}{q}$ for some $j$,
      \[
\text{Arclength of} \ F \ \text{on} \ [0,1] = \lim_{n\to\infty} \sum_{i=0}^{q^n-1} \sqrt{ \left( \frac{1}{q^n} \right)^2 +  (\Pi_n(v_{i}))^2 } = 2.
      \]
We proved this by showing the distribution is singular. Ideally, we would like a direct proof of this fact to facilitate a deeper understanding of the iteration approach given by the recursive formula in (\ref{seqdef}). Specifically, we would like to rigorously quantify how much additional arclength is witnessed every time we iterate.

\subsection{Future Computational Work}

The largely theoretical results shown inform us about how distributions should look when unfair coins or dice are thrown.  Of immediate computational interest is the formula for arclength after $n$ tosses given in Theorem \ref{theorem:arclength}.  Given precisely tuned devices returning weighted dice or coin flips results, one could feasibly form a sort of hypothesis test on how many flips one must take before noticing a deviation from a uniform distribution.  Verifying the arclength formula in device simulation and devising such a test is the subject of ongoing work.

Beyond the fast application of the comparison metrics, the form of the singular measures themselves provides an inspiration point for future microelectronic design and verification.  In our proof of the folk theorem, Theorem \ref{maintheorem2}, an argument is made based on a set of full measure for one weighting being a set of zero measure for all other weightings.  This result then provides a basis for verifying the distribution of an array of $p$-weighted coin-like devices. If one had the probability measure induced by each device and had the set of all binary numbers with density $p$, then the measure of that set determines if the device is correctly weighted. Obviously, such a distributional object does not exist.  However, this theory touchpoint can guide the discovery of future approximate methods and heuristics.

Future work will see us put these comparative tools into practice via simulation. We will analyze the resulting data, discuss their merits and difficulties, and offer additional refinements to their implementation.